\documentclass[12pt]{article}
\usepackage{amsthm}
\usepackage{amsmath}
\usepackage{amssymb}
\usepackage[affil-it]{authblk}
\begin{document}

\newtheorem{defn}{Definition}
\newtheorem{lemma}{Lemma}
\newtheorem{proposition}{Proposition}
\newtheorem{corollary}{Corollary}
\newtheorem{theorem}{Theorem}
\newtheorem{question}{Question}

\def\seq#1#2{\langle\,#1_{#2}\mid#2\in\omega\,\rangle}
\def\seqn#1#2#3{\langle\,#1\mid#2\in#3\,\rangle}
\def\corn{\mathop{\rm cor}}
\def\so{\mathop{\rm so}}
\def\hgt#1#2{{\rm ht}_{#1}(#2)}
\def\hght#1{{\rm ht}(#1)}
\def\St{{\cal S}}
\def\uG{{\bf 1}_G}
\def\set#1{\sset#1.}
\def\sset#1|#2.{\{\,#1\mid#2\,\}}
\def\nwd{\mathop{\sf nwd}}
\let\rst\upharpoonright
\def\Int{\mathop{\rm Int}}
\let\force\Vdash
\author{Alexander Shibakov \thanks{Electronic address: \texttt{ashibakov@tntech.edu}}}
\affil{Tennessee Technological University}
\title{No interesting sequential groups}
\maketitle
\begin{abstract}
\noindent We prove that it is consistent with ZFC that no
  sequential topological groups of intermediate sequential
  orders exist. This shows that the answer to a 1981 question of P.~Nyikos
  is independent of the standard axioms of set theory. The model
  constructed also provides consistent answers to several
  questions of D.~Shakhmatov, S.~Todor\v cevi\'c and Uzc\'ategui. In
  particular, we show that it is consistent with ZFC that every
  countably compact sequential group is Fr\'echet-Urysohn.
\end{abstract}

\section{Introduction}
A number of areas in mathematics benefit from viewing continuity
through the lens of convergence. To investigate the effects of convergence
on topological structure several classes of spaces have been
introduced and studied by set-theoretic topologists. These range from
various generalized metric spaces to sequential ones. As a result of
these efforts a vast body of classification results and metrization
theorems have been developed. 

A popular theme has been the study of convergence in the presence of an
algebraic structure such as a topological group (see \cite{arhtkac}
and \cite{sha} for a bibliography). One of the first results of this
kind, the classical metrization theorem by Birkhoff and Kakutani states that every first
countable Hausdorff topological group is metrizable. This establishes a
rather unexpected connection between the local and the global properties
generally unrelated to each other.

It has been demonstrated by a number of authors, however, that various
shades of convergence are, in general, different from each other, even
when an algebraic structure is involved (see \cite{arh}, \cite{nyikos},
\cite{sha1}, \cite{shi}). A common thread among the majority of these results
is the necessity of set-theoretic assumptions beyond ZFC to construct
counterexamples.

A celebrated solution of Malykhin's problem about the metrizability of
countable Fr\'echet groups by Hru\v s\'ak and Ramos-Gars\'ia \cite{hrusak} is a
beautiful validation of the significance of set-theoretic tools in the
study of convergence.

A question that is only slightly more recent than Malykhin's problem
was asked by P.~Nyikos in \cite{nyikos} and deals with the {\em
  sequential order} in topological groups. Recall that a space $X$ is
{\em sequential} if whenever $A\subseteq X$ is not closed, there
exists a convergent sequence $C\subseteq A$ that converges to a point
outside $A$. This is a rather indirect way of saying that convergent
sequences determine the topology of $X$ without supplying any
`constructive means' of describing the closure operator. Such a
description is provided by the concept of the {\em sequential closure}
of $A$: $[A]' = \text{``limits of all convergent sequences in
  $A$''}$. The next natural step is to recursively define {\em
  iterated} sequential closures $[A]_\alpha$, $\alpha\leq\omega_1$ as
$[A]_{\alpha+1} = [[A]_\alpha]'$ and $[A]_\alpha =
\bigcup\set{A_\beta|\beta < \alpha}$ for limit $\alpha$. It is
a quick observation that in sequential spaces $[A]_{\omega_1} = \overline{A}$
and in fact, this property characterizes the class of sequential spaces.

For many spaces it takes only countably many iterations to get the
closure of any set. The smallest ordinal $\alpha\leq\omega_1$ such
that $[A]_\alpha = \overline{A}$ for every $A\subseteq X$ is called
{\em the sequential order} of $X$ which is written $\alpha = \so(X)$. As a
simple illustration of these concepts, sequential spaces are those for
which the sequential order is defined and Fr\'echet (or
Fr\'echet-Urysohn) ones are those whose sequential order is $1$.

Simple examples of spaces of {\em intermediate} (i.e.~strictly between $1$ and
$\omega_1$) sequential orders are plentiful but they all seem to have
one common feature: different points of the space have different
properties in terms of the sequential closure. This led P.~Nyikos to
ask the following question.

\begin{question}[\cite{nyikos}]Do there exist topological groups of
  intermediate sequential orders?
\end{question}

This question and some of its stronger versions were also asked by
D.~Shakhmatov in \cite{sha} (Questions 7.4 (i)--(iii)).

A weak version of this question (for homogeneous and {\em
  semi}$\,$topo\-logi\-cal groups, i.e.~groups in which the multiplication is
continuous in each factor separately) had been answered affirmatively in
ZFC (see \cite{dpierone} and \cite{pierone}).

A consistent positive answer for topological groups was first given
in \cite{sh} using CH. In \cite{shi} it was shown that under CH groups
of every sequential order exist.

In this paper we use some of the techniques developed by Hru\v s\'ak and
Ramos-Garc\'ia for their solution of Malykhin's problem to show that
extra set-theoretic assumptions are necessary. To be more precise, we
construct a model of ZFC in which all sequential groups are either
Fr\'echet or have sequential order $\omega_1$. For countable groups,
the result can be viewed as a consistent metrization statement: it is
consistent with the axioms of ZFC that all countable sequential groups
of sequential order less than $\omega_1$ are metrizable.

As an aside, we show how the same model provides a consistent answer
to a question of D.~Shakhmatov about the structure of countably
compact sequential groups.

\section{Preliminaries}
We use standard set-theoretic terminology, see \cite{kunen}. By a slight
abuse of notation we sometimes treat sequences as sets of points in
their range. Basic facts about topological groups can be found in
\cite{arhtkac}. All spaces are assumed to be regular.

Following \cite{hrusak} define {\em Laver-Mathias-Prikry forcing
  ${\Bbb L}_{\cal F}$} associated to a free filter ${\cal F}$ on
$\omega$ as the set of those trees $T\in\omega^{<\omega}$ for which
there is an $s_T\in T$ (the {\em stem of $T$}) such that for all $s\in
T$, $s\subseteq s_T$ or $s_T\subseteq s$ and such that for all $s\in
T$ with $s\supseteq s_T$ the set $\mathop{\rm succ}_T(s) =
\set{n\in\omega|s^\frown n\in T}\in{\cal F}$ ordered by inclusion.

Full details of proofs of various properties of ${\Bbb L}_{\cal F}$
can be found in \cite{hrusak}, here we only present the statements
directly used in the arguments in this paper.

A central role in the techniques of \cite{hrusak} is played by the
concept of an {\em $\omega$-hitting family}. Recall that a family
${\cal H\subseteq[\omega]^\omega}$ is called $\omega$-hitting
\cite{dow} if given $\seq An\subset[\omega]^\omega$ there is an
$H\in{\cal H}$ such that $H\cap A_n$ is infinite for all
$n\in\omega$. 

The following two statements from \cite{brendle} supply all the
necessary information to establish the preservation of
$\omega$-hitting families by ccc forcings and their iterations.

\begin{lemma}[\cite{brendle}]\label{fsohp}Finite support iterations of ccc forcings
  strongly preserving $\omega$-hitting strongly preserve
  $\omega$-hitting.
\end{lemma}

As noted in \cite{hrusak} a forcing that strongly preserves
$\omega$-hitting preserves $\omega$-hitting. Moreover, as the lemma
below implies, these two concepts are equivalent for the forcings used
in the arguments below so the definition of strong preservation is
omitted.

\begin{proposition}[\cite{brendle}]\label{sohp}Let ${\cal I}$ be an
  ideal on $\omega$ and let ${\cal F} = {\cal I}^*$ be the dual
  filter. Then the following are equivalent.
\begin{itemize}
\item[{\rm(1)}] For every $X\in{\cal I}^+$ and every ${\cal J}\leq_K{\cal
  I}\rst X$ the ideal ${\cal J}$ is not $\omega$-hitting.

\item[{\rm(2)}] ${\Bbb L}_{\cal F}$ strongly preserves $\omega$-hitting.

\item[{\rm(3)}] ${\Bbb L}_{\cal F}$ preserves $\omega$-hitting.

\end{itemize}
\end{proposition}

The {\em Kat\v etov order $\leq_K$} on ideals used above
is defined by putting ${\cal I}\leq_K{\cal J}$ whenever ${\cal I}$ and
${\cal J}$ are ideals on $\omega$ and there exists an
$f:\omega\to\omega$ such that $f^{-1}(I)\in{\cal J}$ for every
$I\in{\cal I}$.

Let $X$ be a topological space, $\nwd(X)$ be the ideal of nowhere dense
subsets of $X$, and $\nwd^*(X)$ be the filter of dense open subsets of
$X$. Let ${\cal I}_x$ be the ideal dual to the filter of open neighborhoods of $x\in
X$. Finally, the {\em $\pi$-character} $\pi\chi(x, X)$ is defined as
the smallest cardinality of a family ${\cal U}$ of open subsets of $X$
such that every neighborhood of $x$ contains a $U\in{\cal U}$. It is a
well known fact that $\pi$-character and character coinside in
topological groups thus every nonmetrizable topological group has an
uncountable $\pi$-character due to Birkhoff-Kakutani theorem.

The next proposition is a direct restatement of Proposition~5.3 (a)
\cite{hrusak}. 

\begin{proposition}[\cite{hrusak}]\label{pihit}
Let $X$ be a countable regular space and $x\in X$. If $\pi\chi(x, X) >
\omega$ then $\force_{{\Bbb L}_{\nwd^*(X)}}\text{``}\dot A_{gen}\in{\cal
  I}_x^+\wedge{\cal I}_x\rst\dot A_{gen}\text{ is
  $\omega$-hitting}\text{''}$
\end{proposition}

Lemmas~6.4 and~6.5 in \cite{hrusak} are stated for a countable
Fr\'echet group $G$ whereas a careful reading of their proofs reveals that
Fr\'echetness can be replaced by a weaker condition below.

\begin{itemize}
\item[(C)] For every countable family $\set{N_i|i\in\omega}$ of
  nowhere dense subsets of $G$ there exists a nontrivial convergent
  sequence $C\subseteq G$ such that $C\to\uG$ and $C\cap N_i$ is
  finite for every $i\in\omega$.

\end{itemize}

To state the lemma we need two more definitions from
\cite{hrusak}. The definitions encapsulate the connection between the
algebraic structure of an abstract group $G$ and its topology.

\begin{defn}[\cite{hrusak}]Let $G$ be an abstract group and let $X\subseteq
  G\setminus\{\uG\}$. A subset $A\subseteq G$ is called {\em
    $X$-large} if for every $b\in X$ and $a\in G$ either $a\in A$ or
  $b\cdot a^{-1}\in A$.
\end{defn}

\begin{defn}[\cite{hrusak}]A family ${\cal C}$ of subsets of an abstract group $G$ is
  {\em $\omega$-hitting w.r.t.~$X$} if given a family $\seq An$ of $X$-large
  sets there is a $C\in{\cal C}$ such that $C\cap A_n$ is infinite for
  all $n\in\omega$.
\end{defn}

\begin{lemma}[\cite{hrusak}]\label{closowh}Let $G$ be a topological
  group that satisfies {\rm (C)}. Then
$$
\force_{{\Bbb L}_{\nwd^*(G)}}\text{``$\,{\cal C}$ is $\omega$-hitting
  w.r.t.~$\dot A_{gen}$''}
$$
where ${\cal C} = {\cal I}^\perp_{\uG}$ is the ideal consisting of
sequences converging to $\uG$.
\end{lemma}

The last two lemmas from \cite{hrusak} deal with the preservation of
$\omega$-hitting w.r.t.~$X$. As before, \cite{hrusak} notes that
strong preservation implies preservation, thus the definition of
strong preservation of $\omega$-hitting w.r.t.\ $X$ is omitted.

\begin{lemma}[\cite{hrusak}]\label{sigpr}$\sigma$-centered forcings strongly preserve $\omega$-hitting
  w.r.t.~$X$.
\end{lemma}

\begin{lemma}[\cite{hrusak}]\label{fipr}Finite support iterations of ccc forcings that strongly
  preserve $\omega$-hitting w.r.t.~$X$ strongly preserve
  $\omega$-hitting w.r.t.~$X$.
\end{lemma}

Recall that a space $X$ is called $T_5$ (or {\em hereditarily normal})
if every subspace of $X$ is normal. For every $x\in X$, the
{\em pseudocharacter} $\psi(X,x)$ of $x$ in $X$ is defined as the
smallest cardinality of a family ${\cal U}$ of open neighborhoods of
$x$ such that $\bigcap{\cal U} = \{x\}$. The pseudocharacter of $X$ is
then $\psi(X) = \sup\set{\psi(X,x)|x\in X}$.

The following result from \cite{dikranjan} is an elegant extension
of Kat\v etov's product lemma to topological groups.

\begin{theorem}[\cite{dikranjan}]\label{hnpsi}Let $G$ be a $T_5$
  topological group. If there exists a nontrivial convergent sequence
  in $G$ then $\psi(G) = \omega$.
\end{theorem}

\section{Convergence and scaffolds}
While convergent sequences determine the topology of a sequential
space, a more precise description of the closure operator will be
needed later. This descripion is supplied by the idea of a {\em
  scaffold}, defined below. This is not the only, or even the most
efficient way of studying the sequential closure, simply one that
suits the approach below.

\begin{defn}\label{sc} Let $X$ be a topological space, $S\subseteq
  X$ and $\St\subseteq 2^S$. Then {\em $(S,\St)$ is an
    $\alpha$-scaffold} (or simply {\em scaffold}), $\hght S$, $\hgt
  Sx$ for $x\in S$ and {\rm $\corn S$} are defined recursively as:
\begin{itemize}
  \item[\rm(S.0)] If $S = \{x\}$, where $x\in X$, and
  $\St = \{S\}$. Then $(S,\St)$ is a $0$-scaffold $\hght S = 0$, $\hgt Sx =
  0$, and $x = \corn S$.

  \item[\rm(S.1)] Suppose there exist an $x\in S$, a
  disjoint collection $\seq Un$ of open subsets of $X$, and
  $\alpha_n$-scaffolds $(S_n,\St_n)$ such that $\corn S_n\to x$, 
  $\overline{S_n}\subseteq U_n$, $x\not \in\overline{U_n}$ where $\seq \alpha n$
  is non-decreasing and $\alpha = \min\set{\beta|\beta>\alpha_n\text{ for each
      $n\in\omega$}}$. Suppose also that $\St = \{S\}\cup\seq \St n$. 
  Then $(S,\St)$ is an $\alpha$-scaffold, $\hght S = \alpha$, $\hgt Sx =
  \alpha$, $\hgt S{x'} = \hgt {S_n}{x'}$ for $x'\in S_n$ and $\corn S=x$.
\end{itemize}
\end{defn}

As one would expect, most proofs involving scaffolds proceed by
tedious induction arguments on the scaffold's height. Given a scaffold
$(S,\St)$ it will be convenient to define some subsets of $S$ and
$\St$ to simplify the notation. Put $S_{[\beta]} = \set{s\in S|\hgt
  Ss = \beta}$, $\St_{[\beta]} = \set{{\cal T}\in \St|\hght {\cal T}
  = \beta}$.

A subset $S$ of $X$ will be called a(n) ($\alpha$-)scaffold if there
exists a $\St\in2^S$ (called the {\em stratification of $S$}) such
that $(S,\St)$ is an ($\alpha$-)scaffold.

While a stratification is not unique, $\corn S$, $\hgt S x$, $\hght S$, as
well as $S_{[\beta]}$ and similar subsets, are independent of the
choice of $\St$. This is most easily observed by noting that all of
the ordinals and subsets in the list above can be expressed in terms
of the {\em Cantor-Bendixon rank} (see~\cite{kannan}).

The utility of scaffolds is illustrated by the lemma below.
\begin{lemma}\label{econv}Let $X$ be a regular space,
  $x\in[A]_\alpha\subseteq X$. Then there exists a $\beta$-scaffold
  $(S,\St)$, such that $\beta\leq\alpha$, $S\subseteq X$,
  $S_{[0]}\subseteq A$, and $x = \corn S$.
\end{lemma}

If $\St$ is a stratification of
$S$ and $S'\in\St$ put $\St|_{S'}=\{\,T\in\St\mid T\subseteq S'\,\}$.
Then $S'$ is a scaffold and $(S',\St|_{S'})$ is an $\hght {S'}$-scaffold. 
For every stratification $\St$ of $S$,
the collection $\St^- = \{\,S'\in \St, S' \not= S\mid S'\subseteq
S''\in\St\Rightarrow S'' = S' \text{ or } S'' = S\,\}$ of $\subseteq$-maximal
elements of $\St\setminus\{S\}$ always satisfies (S.1) and for each
$T\in\St$ there exists a unique $T^+(\St)\in\St$ such that $T\in\St|_{T^+(\St)}^-$.

By recursively `trimming' $U_n$'s in (S.1) one can prove the following lemma.
\begin{lemma}Let $(S,\St)$ be a scaffold in some space $X$. Then there
  exists a mapping $o:\St\to\tau$, where $\tau$ is the topology on
  $X$, such that $\overline{T}\subseteq o(T)$ and $T'\subseteq T''$ if and
  only if $o(T')\subseteq o(T'')$.
\end{lemma}

We will call such a mapping (or, sometimes, just its set of values) {\em
an open stratification of $S$}.

The scaffold whose existence is provided by Lemma~\ref{econv} is not
always suitable for the purposes of the argument and sometimes has to be
`thinned'. The definition below makes this idea precise.

\begin{defn}
Let $S$ be an $\alpha$-scaffold, $\St$ be a stratification of $S$ and
$S'\subseteq S$. Call $S'$ {\it an $\St$-proper subscaffold\/} (or
simply {\it a proper subscaffold} if $\St$ is of no importance) of
$S$ and write $S'\leq_\St S$
if $\corn S = \corn S'$ and there exists a (unique) stratification $\St'$
of $S'$ such that $\St'^-\neq\emptyset$ whenever $\St^-\neq\emptyset$,
and each $\beta$-scaffold $B'\in\St'^-$ is an
$\St|_B$-proper subscaffold of some $\beta$-scaffold $B\in\St^-$.
\end{defn}

The lemma below introduces a general procedure for picking a
subscaffold inside a given scaffold. It has a standard inductive proof
which is therefore omitted.

\begin{lemma}\label{thn}Let $(S,\St)$ be a scaffold in some space $X$, $\tau$ be
  the topology of $X$, and $r:X\to\tau$ be a neighborhood assignment
  such that $x\in r(x)$ for every $x\in X$. Then there exists an $\St$-proper
  subscaffold $(S',\St')$ of $(S,\St)$ and its open stratification
  $o':\St'\to\tau$ such that $o'(T)\subseteq r(\corn T)$. 
  Moreover, if $o:\St\to\tau$ is any open stratification of $S$,
  the stratification $\St' = \set{o(T)\cap S'|T\in\St, o(T)\cap
    S'\neq\emptyset}$, and a subfamily of
  $\set{o(T)|T\in\St}$ forms an(other) open stratification of $S'$. 
\end{lemma}

The following definition and the subsequent lemma describe how
scaffolds can be used to gauge the sequential order of a space.
\begin{defn}
An $\alpha$-scaffold $S$ is called {\it semiloose\/} (in $X$ where $X$
is a topological space such that $S\subseteq X$) if for every
infinite convergent sequence $C\subseteq S$ such that $C\to s\in S$
there exists an infinite subsequence $C'\subseteq C$ so that 
$\hgt Ss = \min\set{\beta|\beta>\hgt Sx\text{ for all but finitely
    many $x\in C'$}}$. 
\end{defn}

Note in the lemma below that to establish a lower bound on the
sequential order a closed scaffold is needed.
\begin{lemma}Let $X$ be a regular space, $A\subseteq X$,
  $x\in[A]_\alpha$, and $x\not\in[A]_\beta$ for any $\beta <
  \alpha$. Then there exists a semiloose $\alpha$-scaffold $(S,\St)$
  in $X$ such that $x = \corn S$ and $S_{[0]}\subseteq A$. If
  $(S,\St)$ is a closed semiloose $\alpha$-scaffold in $X$ then $\corn
  S\in[S_{[0]}]_\alpha$ and $\corn S\not\in[S_{[0]}]_\beta$ for any
  $\beta < \alpha$.
\end{lemma}

$S$ is called {\it loose\/} if for every $s\in S$ there exists a
convergent sequence $C_s\subseteq S$ such that $C\subseteq^* C_s$ for
some $s\in S$ for any convergent sequence $C\subseteq S$.
Note that loose and semiloose coinside for finite $\alpha$'s. We will
not use loose scaffolds below, they are defined here merely to justify
the choice of terminology.

\section{Scaffolds in topological groups}\label{scissgm}
To a large extent, the central part of the argument below aims to
establish property (C) for sequential topological groups satisfying
some conditions. In \cite{hrusak} it is noted that (C) holds for
Fr\'echet spaces without isolated points due to a lemma in
\cite{dowbarman}. This property does not hold for arbitrary sequential
groups (a quick example is provided by the free (Abelian) topological
group over a convergent sequence) so some additional restrictions are
necessary.

The following definition is introduced to facilitate the study of
sequential order in sequential topological groups. $1$-scaffolds 
(i.e.~convergent sequences) are treated separately as they form an important special case.

\begin{defn}\label{tsrs}Let $G$ be a topological group, $C = \seq
  cn\cup\{c\}\subseteq G$ be a convergent sequence in $G$, $c_n\in
  U_n$ be disjoint open subsets of $G$, and $c\not\in\overline{U_n}$
  for any $n\in\omega$ (i.e. $\set{U_n|n\in\omega}\cup\{G\}$ is an
  open stratification of $C$). Let $(S,\St)$ be a scaffold in $G$ and
  $\St^- = \set{S_n|n\in\omega}$. Suppose $\overline{c_n\cdot S_n}\subseteq
  U_n\cdot\corn S$ for every $n\in\omega$. Define the scaffold
$$
C\otimes S = \bigcup\set{c_n\cdot S_n|n\in\omega}\cup\{c\cdot\corn
S\}
$$
Also put $C\odot S = C\otimes S\setminus\corn (C\otimes S) =
\bigcup\set{c_n\cdot S_n|n\in\omega}$.
\end{defn}

Note that the definition of $C\otimes S$ depends on the indexing of
the elements of $\St^-$ and $C$, as well as the choice of $U_n$'s. The
latter is rarely a problem since any argument involving $\otimes$'s is
usually preceded and followed by passing to appropriate
subscaffolds. In particular, in order to satisfy $\overline{c_n\cdot S_n}\subseteq
  U_n\cdot\corn S$, observe that $c_n\cdot\corn S\in U_n\cdot\corn S$,
  and construct open nested $V_n$'s such that $\corn S\in V_n$ and
  $\overline{c_n\cdot V_n}\subseteq U_n\cdot\corn S$. Now pick increasing
  $n(i)\in\omega$ so that $\corn S_{n(i)}\in V_i$. `Thin' each
  $S_{n(i)}$ using Lemma~\ref{thn} to obtain proper subscaffolds
  $S_i'\subseteq S_{n(i)}$ such that $\overline{S_i'}\subseteq
  V_i$. Put $S' = \bigcup\set{S_i'|i\in\omega}\cup\{\corn S\}$. Now
  $C\otimes S'$ is defined. Note that for every proper subscaffold
  $S''\subseteq S'$ of $S'$ defined as above, the product $C\otimes
  S''$ is also defined. 

The indexing dependence can be mitigated by requiring
that both factors be ordered in the type of $\omega$ and the products
be taken `in order'. The definition of proper subscaffold can be
adjusted to inherit the order, as well.

The next lemma shows that $\otimes$ does not introduce new convergent
sequences.

\begin{lemma}\label{tsr1}Let $G$ be a topological group, $(S,\St)$ and $C$ be as
  in Definition~\ref{tsrs}. Let $C' = \seqn {c_{n(i)}\cdot
    s_i}i\omega\subseteq C\odot S$ be a convergent sequence. Then
  $\seq si$ is a convergent sequence.
\end{lemma}

The remarks following Definition~\ref{tsrs} fully apply to the general
case defined below. Disambiguating measures suggested there
(e.g.~ordering of $S$ and $S_n$'s) are assumed to be taken but are not
explicitly mentioned.

\begin{defn}\label{tsr}Let $G$ be a topological group. Let $(S,\St)$ be an
  $\alpha$-scaffold, $\alpha>1$, and
  $\set{(S_n,\St_n)|n\in\omega}$ be a countable collection of
  scaffolds in $G$, the sequence of $\hght {S_n}$'s is
  nondecreasing and $\set{\corn S_n|n\in\omega} = \{x\}$ for some $x\in G$. Let also
  $\set{C_n|n\in\omega}$ list all $C_n\in\St$ such that $\hght {C_n} =
  1$. Suppose $C_n\otimes S_n$
  are defined for every $n\in\omega$ and the natural (i.e.~taken from
  some open stratification of $S$) choice of open $U_m^n$ as
  in Definition~\ref{tsrs}. Put
$$
S\otimes\set{S_n|n\in\omega} = (S\setminus S_{[0]})\cdot
x\cup\bigcup\set{C_n\odot S_n|n\in\omega}
$$
\end{defn}

The next lemma will not be used explicitly in what follows. Rather it
presents an induction hypothesis that can be used to justify the claim
that $\otimes$-products result in scaffolds in the general case.

\begin{lemma}\label{tsrd}Let $(S',\St')$ and $(S'',\St'')$ be scaffolds
  and $\hght{S'}\leq\hght{S''}$. Let $\set{S_n|n\in\omega}$ be a
  family of (ordered) scaffolds so that the sequence of $\hght{S_n}$'s is
  nondecreasing. Let $\seq mi\subseteq\omega$ and $\seq ni\subseteq\omega$ be arbitrary
    increasing sequences so that both $S'\otimes\set{S_{n_i}|i\in\omega}$ and
  $S''\otimes\set{S_{m_i}|i\in\omega}$ are defined. Then both
    $S'\otimes\set{S_{n_i}|i\in\omega}$ and
    $S''\otimes\set{S_{m_i}|i\in\omega}$ are scaffolds and
    $\hght{S'\otimes\set{S_{n_i}|i\in\omega}}\leq\hght{S''\otimes\set{S_{m_i}|i\in\omega}}$. Moreover, 
$\hght{S'\otimes\set{S_{n_i}'|i\in\omega}} =
    \hght{S'\otimes\set{S_{n_i}|i\in\omega}}$ whenever
    $\set{S_n'|n\in\omega}$ is a sequence of proper subscaffolds of $S_n$'s.
\end{lemma}

The lemma and the corollary that follow express the idea that one
can only raise the height by following a convergent sequence in a
scaffold.

\begin{lemma}Let $(S,\St)$ be a scaffold in some space $X$. Let $C =
  \seq cn\subseteq S$ be a converging sequence in $S$ such that
  $c_n\to c\in S$ and $c_n\in T_n\in\St$ where $T_n$'s are
  disjoint. Then $c = \corn T$ for some $T\in\St$ and $T$ contains all
  but finitely many $T_n$'s.
\end{lemma}

\begin{corollary}\label{mooff}Let $(S,\St)$ and
  $\set{(S_n,\St_n)|n\in\omega}$ be such that
  $T = S\otimes\set{S_n|n\in\omega}$ is well-defined. Let $C = \seq
  ci\subseteq S\otimes\set{S_n|n\in\omega}$ be a convergent sequence
  such that $c_i\in C_{n_i}\odot S_{n_i}$ (here we reuse the
  notation from Definition~\ref{tsr}) for increasing $n_i$'s and
  $c_i\to c\in T$. Then $c = s\cdot x$ where $x$ is the same as in
  Definition~\ref{tsr} and $\hgt Ss\geq2$.
\end{corollary} 

Recall that $\frak b$ is the smallest cardinality of an unbounded
family in $\omega^\omega$. The lemma and the corollary below are
probably folklore, however, the author could not find a reference for
the general form of this fact. For a proof of a similar statement
about 2-scaffolds, see, for example~\cite{noguratanaka}.

\begin{lemma}\label{cscaff}Let $(S,\St)$ be an $\alpha$-scaffold, 
  $\cal U$ be a collection of open subsets of
  $X$ such that $|U| < \frak b$. Then there exists an $S'\leq_{\St}S$
  such that every open neighborhood $U\in\cal U$ of $\corn S'$
  contains all but finitely many $T\in\St'^-$, where $\St'$ is the
  stratification of $S'$ that witnesses $S'\leq_{\St}S$.
\end{lemma}
\begin{proof}
Suppose the statement is true for all $\beta$-scaffolds with
$\beta<\alpha$. Let $\St^- = \seq Sn$ and $\St|_{S_n}^-
= \seq {S^n} m$. 

Modify $(S_n,\St|_{S_n})$ if necessary using the inductive hypothesis, 
and for each $U\in\cal U$ construct a function $f_U:\omega\to\omega$ such that $\corn
S_n\in U$ implies $S^n_m\subseteq U$ for $m > f_U(n)$. Let
$f:\omega\to\omega$ be a function that dominates $\{\,f_U\mid
U\in{\cal U}\,\}$. Put $\St'=\bigcup \{\,\St|_{S^n_m}\mid m >
f(n)\,\}\cup \seq {S'} n\cup S'$ where $S'_n = \bigcup \{\,S^n_m\mid m >
f(n)\,\}\cup\{\corn S_n\}$ and $S' = \bigcup S'_n\cup\{\corn S\}$. It
is straightforward to check that $\St'$ is a stratification of $S'$
that witnesses $S'\leq_{\St}S$ and that $\St'^- = \seq {S'} n$.

Let now $\corn S\in U\in\cal U$ and pick an $n'\in\omega$ be such that $f(n)
> f_U(n)$ and $\corn S_n\in U$ for all $n > n'$. If $n > n'$ it
follows from the definition of $S'_n$ that $S'_n\subseteq U$. 
\end{proof}

\begin{corollary}\label{cscaffs}Let $Y$ be a regular space, $X\subseteq Y$ be such
  that $\psi(X) < \frak b$, $(S,\St)$ be an $\alpha$-scaffold. 
  Then there exists an $S'\leq_{\St}S$ closed in $X$.
\end{corollary}
\begin{proof}
Suppose the statement is true for $\beta$-scaffolds such that
$\beta<\alpha$, so assume that $S$ and its stratification $\St$ are such
that each element of $\St^-$ is closed in $X$. Note that a proper
subscaffold of a scaffold closed in $X$ is also closed in $X$.

Using the regularity of $Y$ and $\psi(X,\corn S) < \frak b$ pick a
family $\cal U$ of open subsets of $Y$ such that $|\cal
U| < \frak b$, $\corn S\in U$ for every $U\in\cal U$ and
$\bigcap\{\,\overline U^Y \mid U\in{\cal U}\,\}\cap X = \{\corn S\}$.
Apply Lemma~\ref{cscaff} to construct
$S'\leq_{\St}S$ and $\St'$. 

Suppose $\overline{S'}\ni x\in X\setminus S'$ for some $x\in X$. Let
$U\in\cal U$ be such that $\corn S\in U\subseteq\overline U\subseteq
Y\setminus\{x\}$. Since all but finitely many elements of $\St'^-$ are
subsets of $U$, $x\in\overline {T'}$ for some $T'\in \St'^-$. Now
$T'\subseteq T$ for some $T\in \St^-$. This (and $S'\leq_{\St}S$) implies that $T'$ is
closed in $X$ so $x\in T'\subseteq S'$ contrary to the assumption above.
\end{proof}

\begin{lemma}\label{cvr}Let $S$ be an $\alpha$-scaffold and let $\cal A$ be a
  cover of $S_{[0]}$. Then either there is an $A_S\in{\cal A}$ such that $\overline{A_S\cap
    S_{[0]}}\ni \corn S$ or there exists a countable family ${\cal
    A}_S\subseteq[{\cal A}]^\omega$ such that any ${\cal
    A}'\subseteq{\cal A}$ with the property $|{\cal A}'\cap {\cal
    A}''| = \omega$ for every ${\cal A}''\in{\cal A}_S$ satisfies
  $\overline{\bigcup\{\,S_{[0]}\cap B\mid B\in{\cal A}'\,\}}\ni\corn S$.
\end{lemma}
\begin{proof}
Suppose the statement is true for all $\beta$-scaffolds such that
$\beta<\alpha$. For each $T\in\St^-$ pick either an $A_T\in \cal A$
such that $\corn T\in\overline{T_{[0]}\cap A_T}$ or a countable family ${\cal
  A}_T\subseteq [{\cal A}]^\omega$ such that any ${\cal
  A}'\subseteq{\cal A}$ with the property $|{\cal A}'\cap{\cal
  A}''|=\omega$ for every ${\cal A}''\in{\cal A}_T$ satisfies
$\overline{\bigcup\{\,T_{[0]}\cap B\mid B\in{\cal A}'\,\}}\ni\corn T$.

If there is a single $B$ such that $B = A_T$ for infinitely many $T$'s
as above, put $A_S = A_T$. If there are infinitely many different
$A_T$'s with the property above put ${\cal A}_S = \{\,\{\,A_T\mid
T\in\St^-\text{ and $A_T$ is defined as above }\,\}\,\}$. Otherwise
put ${\cal A}_S = \bigcup\{\,{\cal A}_T\mid T\in\St^-\text{ and ${\cal
    A}_T$ is defined as above }\,\}$.
\end{proof}

Recall that a continuous map $p:X\to Y$ is called {\em hereditarily
  quotient} if $p$ is quotient (i.e.~$U\subseteq X$ is open if and
only if $p^{-1}(U)$ is open) and for any $A\subseteq X$, the
restriction $p|_{p^{-1}(A)}:p^{-1}(A)\to A$ is also quotient. It is
straightforward that every open map is hereditarily quotient. The
following lemma presents a well known property of such maps.

\begin{lemma}[\cite{kannan}] Let $p:X\to Y$ be a hereditarily quotient
  map. Then $\so(Y)\leq\so(X)$.
\end{lemma}

Since the natural quotient map from a group onto its quotient is open,
we have that the sequential order of a group cannot be raised by
taking a quotient. Note that for groups there is a more direct proof of this
result, by `lifting' every scaffold in $G/N$ to $G$.

\begin{corollary}\label{loord}Let $G$ be a sequential topological group, $N\subseteq G$
  be a closed normal subgroup of $G$. Then $\so(G/N)\leq\so(G)$.
\end{corollary}

\begin{lemma}\label{ored}Let $G$ be a sequential topological group, $\so(G) <
  \omega_1$, and every scaffold in $G$ have a closed proper
  subscaffold. Let $Y\subseteq G$ be a subgroup, $X = U\cap Y$ for
  some open $U\subseteq G$. Suppose for
  every semiloose $\beta$-scaffold $S\subseteq G$, where
  $\beta\leq\so(G)$ there exists a semiloose $\beta$-scaffold $S'$ in
  $G$ such that $S'\subseteq Y$. Then for every $g\in \overline{X}$ there exists
  a sequence of points of $X$ converging to $g$.
\end{lemma}
\begin{proof}
Let $g\in G$ be arbitrary. Since $\so(G) < \omega_1$ and $g\in\overline{X}$,
there exists a scaffold $(S,\St)$ such that $S\subseteq G$,
$S_{[0]}\subseteq X$, and $g = \corn S$. Define $h(g)$ to be the smallest
$\hght S$ among all such scaffolds. Let $(S,\St)$ be a scaffold
witnessing $\hght S = h(g)$. Using induction, replacing parts of $S$,
and going to proper subscaffolds, if necessary, we can assume that
$\hgt Sq = h(q)$ for every $q\in S$. Suppose $h(g) = \hght S > 1$.

Consider a limit $\alpha = \so(G)$ (the nonlimit case is
similar) and choose a countable
family $\set{S_n|n\in\omega}$ of semiloose $\alpha_n$-scaffolds such
that $S_n\subseteq Y$, $\corn {S_n} = \uG$ and $\alpha_n = \hght {S_n}$ is increasing so
that $\lim\alpha_n = \alpha$. Pick a closed proper subscaffold
$T\subseteq S\otimes\set{S_n|n\in\omega}$ and let $A = T_{[0]}$. Since
$U$ is open in $G$ and $S_{[0]}\subseteq U$, we may assume $A\subseteq
U$. Since $Y$ is a group, $A\subseteq Y$ so $A\subseteq X$. In
addition, $g\in\overline{A}$ and 
$T\setminus X = S\setminus X = \set{s\in S|\hgt Ss\geq1}$.

Let $\beta$ be the smallest ordinal such that $t\in[A]_\beta$ for some
$t\in T\setminus X$. Since there is a sequence of points of $X$
converging to $t$, $\hgt St = 1$ by the choice of $S$. Let $C = \seq
ci$ be such a sequence. If $c_i\in C_{n_i}\otimes S_{n_i}'$ for an
increasing $\seq ni$ (here we borrow the terminology of
Definition~\ref{tsr}) then by Corollary~\ref{mooff} $\hgt St > 1$
contrary to the choice of $S$. Thus we may assume that $C\subseteq
C_n\otimes S_n'$ for some $n\in\omega$.

A similar argument and induction on $\hgt {C_n\otimes S_n'}c$ shows that
$t\in[A\cap C_n\otimes S_n']_\beta$. Now, the condition that each
$S_n'$ (therefore, each $C_n\otimes S_n'$ by Lemma~\ref{tsr1}) is
semiloose implies that $\beta \geq\alpha_n$.

Let $\gamma <\alpha$. The argument above shows that (due to
$\seq\alpha n$ increasing) $[A]_\gamma\setminus X$ is finite. Thus
$g\not\in[A]_\gamma$ (otherwise $\hgt Sg = 1$) contradicting $\so(G) =
\alpha$.
\end{proof}

The aim of the next lemma is to establish Property~(C) for some sequential groups.

\begin{lemma}\label{cseq}Let $G$ be a sequential topological group, $\so(G) <
  \omega_1$, and every scaffold in $G$ have a closed proper
  subscaffold. Let $X\subseteq G$ be a dense subgroup. Suppose for
  every semiloose $\beta$-scaffold $S\subseteq G$, where
  $\beta\leq\so(G)$ there exists a semiloose $\beta$-scaffold $S'$ in
  $G$ such that $S'\subseteq X$. Let
  $\set{N_i|i\in\omega}\subseteq2^G$ be a family of nowhere dense
  subsets of $G$. Then there exists a nontrivial $C\subseteq X$ such that
  $C\to\uG$, and $C\cap N_i$ is finite for every $i\in\omega$.
\end{lemma}
\begin{proof}
We can assume that each $N_i$ is closed in $G$. Suppose there is no
nontrivial convergent sequence $C\subseteq X$, $C\to\uG$ such that each $C\cap N_i$ is
finite. 

Suppose first that there is no such $C\subseteq G$.
Pick a nontrivial $C' = \seq cn\to\uG$, put $Y = G$, $U =
G\setminus\bigcup\set{N_i|i\leq n}$. Now $c_n\in\overline{U}$, so one
can apply Lemma~\ref{ored} to find a converging sequence $C_n\to c_n$
such that $C_n\setminus\{c_n\}\subseteq
G\setminus\bigcup\set{N_i|i\leq n}$.
Thinning out the resulting set, if
necessary, we can assume that $T = \bigcup\set{C_n|n\in\omega}\cup\{\uG\}$
is a closed $2$-scaffold.

Suppose $\so(G)$ is a limit ordinal (the nonlimit case is essentially
the same). Let now $\set{S_n|n\in\omega}$ be such that each $S_n\subseteq G$ is a
closed semiloose $\alpha_n$-scaffold where $\alpha_n\to\so(G)$ is
increasing, and $\corn {S_n} = \uG$. Passing to proper subscaffolds if
necessary, assume that $S = T\otimes\set{S_n|n\in\omega}$ is defined
and is a closed scaffold, and for each $n\in\omega$, the set $C_n\odot
S_n\subseteq G\setminus\bigcup\set{N_i|i\leq n}$. The last property
and the assumption at the beginning of the previous paragraph imply that every
sequence of points in $S$ converging to $\uG$ 
contains an infinite subsequence of $C'$. This and the closedness and
semilooseness of $C_n\otimes S_n$ show that
$\uG\not\in[T_{[0]}]_\beta$ whenever $\beta < \alpha_n$ for some
$n\in\omega$ contradicting $\alpha_n\to\so(G)$.

Therefore, we can pick a nontrivial $C' = \seq cn\to\uG$ such that
$c_n\in G\setminus\bigcup\set{N_i|i\leq n}$ for every
$n\in\omega$. Putting $Y = X$, $U = G\setminus\bigcup\set{N_i|i\leq
  n}$, and applying Lemma~\ref{ored} once again we can find a
convergent sequence $C_n\to c_n$ so that $C_n\setminus\{c_n\}\subseteq
X\setminus\bigcup\set{N_i|i\leq n}$. Now the argument in the preceding
two paragraphs can be repeated to produce the desired sequence.
\end{proof}

Intuitively, the iteration argument is set up to eliminate the
unwanted groups in the extension by destroying the appropriate
witnesses in the intermediate stages. It is thus important to keep the
size of the witness small (countable). In the case of Fr\'echet groups
dealt with in \cite{hrusak}, the groups are countable already. In the
case of general sequential groups, the size of the group must be
reduced first. The following definition formalizes one obstacle to
such a reduction.

\begin{defn}\label{colla}Let $H$ be a topological group. Call $H$ {\em
    $\omega_1$-collapsible\/} if for every closed normal subgroup $N$ of $H$,
  $\psi(H/N)\leq\omega_1$ implies $H/N$ is metrizable.
\end{defn}

\begin{lemma}\label{lindelo}Let $H$ be a separable $\omega_1$-collapsible
  topological group. Then every subspace of $H$ of size at most
  $\omega_1$ is Lindel\"of.
\end{lemma}
\begin{proof}
Let $Y\subseteq H$ be an arbitrary subset of $H$ of size at most $\omega_1$.
Let ${\cal U} = \set{U_\alpha|\alpha\in\omega_1}$ be a cover of $Y$ by
open subsets of $H$. For every $y\in Y$ pick an open subset ${\bf
  1}_H\in V_y$ of $H$ so that $y\cdot V_y\cdot V_y\subseteq
U_\alpha$ for some $\alpha\in\omega_1$. Select a countable dense
subgroup $X\subseteq H$. Construct a family ${\cal V} =
\set{V^\alpha|\alpha\in\omega_1}$ of open neighborhoods of ${\bf 1}_H$
so that $V_y\in{\cal V}$ for every $y\in Y$ and for each $V\in{\cal
  V}$ and $x\in X$ there is a $W\in{\cal V}$, $W\cdot W^{-1}\subseteq
V$, $x\cdot W\cdot x^{-1}\subseteq V$. Put $N = \bigcap{\cal V}$. Now
$N$ is a closed subgroup of $H$ that commutes with every $x\in
X$. This, the assumption that $X$ is dense, and the continuity of
$c_a:H\to H$, where $c_a(x) = x\cdot a\cdot x^{-1}$, in conjunction
with the closedness of $N$ imply the (algebraic) normality of $N$.

Let $p: H\to H/N$ be the corresponding quotient map. It follows from
the definition of $\omega_1$-collapsible and the construction of $N$ that $H/N$ is
separable metric. The sets $p(V_y\cdot N)$, $y\in Y$ form an open
cover of $p(Y)$ which has a countable subcover $\set{p(V_{y_i}\cdot
  N)|i\in\omega}$. It is straightforward to see that
$\set{V_{y_i}\cdot N|i\in\omega}$ is a countable open refinement of
${\cal U}$.
\end{proof}

The following lemma will be used to pick small witnesses in some
cases. It is stated for proper forcing notions for the sake of
generality. Its applications in this paper are limited to ccc notions
of forcing only.

\begin{lemma}\label{hlindelo}Let $V$ be a model of CH, ${\Bbb P}\in V$
  be a proper notion of
  forcing, and ${\bf G}$ be ${\Bbb P}$-generic over $V$. Let
  $H\in V[{\bf G}]$ be an $\omega_1$-collapsible topological
  group and $X, G\in V$ satisfy the following properties. $X\subseteq
  G$, $G$ is a topological group algebraically isomorphic to a
  subgroup of $H$ (below we treat $G$ as if it were an actual subgroup of $H$). 
  $X$ is countable subgroup, dense in both $G$ and
  $H$. Furthermore, for every subset $A\subseteq X$, $A\in V$ the
  closures $\overline{A}^G = \overline{A}^H\cap G$. Then $G$ is
  hereditarily Lindel\"of.
\end{lemma}

\begin{proof}
Let $Y\subseteq G$ be any subspace and let $\cal U$ be any open (in $G$)
cover of $Y$. Using CH, the regularity of $G$, and the density of $X$
in $G$, we may assume, after refining ${\cal U}$, if necessary, that ${\cal U}
= \set{O_\alpha = G\setminus\overline{X\setminus U_\alpha}|\alpha\in\omega_1}$
for some $U_\alpha\subseteq X$ open in $X$.

The assumptions about $H$ and $G$ imply that in $V[{\bf G}]$ each
$O_\alpha\cap Y$ is open in $Y$ as a subspace of
$H$. Lemma~\ref{lindelo} implies that there exists a countable (in
$V[{\bf G}]$) subcover ${\cal U}'\subseteq{\cal U}$ of $Y$. Using the
property of proper forcings that countable sets of ordinals in the
extension are subsets of countable sets of ordinals in the ground
model (see~\cite{jech}, Lemma~8.7, for example) and the properness 
of ${\Bbb P}$ concludes the proof.
\end{proof}

\section{Main theorem}

The next lemma is a generalization of Proposition~5.3 (b) from
\cite{hrusak}. The original preservation lemma was stated for
Fr\'echet spaces and could not be reused due to the
lack of an appropriate version of Lemma~5.1(\cite{hrusak}) for the
general case needed here (see the discussion at the beginning of section~\ref{scissgm}).

\begin{lemma}\label{pohit}Let $X\subseteq G$, where $G$ is sequential and $X$ has
  no isolated points as a subspace of $G$. Then ${\Bbb L}_{\nwd^*(X)}$
  strongly preserves $\omega$-hitting.
\end{lemma}
\begin{proof}
We reuse some of the original notation of Proposition~5.3 (b) in \cite{hrusak}.
Let $Y\in\nwd(X)^+$ and suppose there is an $\omega$-hitting ideal
${\cal J}\leq_K\nwd(X)\rst Y$ witnessed by $f:Y\to\omega$. Put $U =
\Int(\overline{Y})\neq\emptyset$ and $Z = U\cap Y$.

Put $N_i = f^{-1}(i)\in\nwd(X)$, ${\cal A} =
\set{N_i|i\in\omega}$. For each $x\in Z$ pick a scaffold 
$S^x\subseteq G$ such that $\corn S^x = x$, $\St_x^-=\set{S_n^x|n\in\omega}$ for some
stratification $\St_x$ of $S^x$ and $(S_n^x)_{[0]}\subseteq
Z\setminus\bigcup\set{N_i|i\leq n}$. Note that ${\cal A}$ is a cover
of $S^x_{[0]}$ so by the choice of $S^x_n$, Lemma~\ref{cvr} implies
the existence of a countable ${\cal A}_x\subseteq[{\cal A}]^\omega$
such that whenever ${\cal A}'\subseteq{\cal A}$ satisfies $|{\cal
  A}'\cap {\cal A}''|=\omega$ for every ${\cal A}''\in{\cal A}_x$, the
point $x = \corn S_x\in\overline{\bigcup\set{S^x_{[0]}\cap B|B\in{\cal
A}'}}$.

Since ${\cal J}$ is $\omega$-hitting there is a $J\in{\cal J}$ such
that $J\cap f(\bigcup{\cal A}'')$ is infinite for every $x\in Z$, ${\cal
  A}''\in{\cal A}_x$. This implies $f^{-1}(J)$ is dense in $Z$ which
is dense in $U$, a contradiction.
\end{proof}

The next definition provides a description of a potential witness of a
sequential group with an intermediate sequential order in the final extension.

\begin{defn}\label{conseqp}Let $X$ be a non-metrizable topological group defined on $\omega$,
  $\overline{\cal S}\subseteq[[X]^\omega]^\omega$. Call $(X,\overline{\cal S})$ {\em
    a consequential pair} if $X$ can be embedded as a subgroup in a
  sequential group $G$ such that  $\so(G) < \omega_1$, for every semiloose
  $\beta$-scaffold in $G$ there exists a semiloose $\beta$-scaffold in
  $X$, and $\overline{\cal S}$ lists every ${\cal T}$ that can be represented as ${\cal
    T}=\set{S\cap X|S\in\St}$ for some scaffold $(S, \St)$ in $G$. If,
  in addition to the properties above, every scaffold in $G$ has a proper
  closed subscaffold, we will call $(X,\overline{\cal S})$ a {\em
    strong} consequential pair.
\end{defn}

Note that such a $G$ is unique up to an isomorphism, since the second
element of the pair uniquely determines both the algebraic structure,
as well as the topology of $G$. We will say that $X$ {\em extends to
  $G$}. 

\begin{theorem}It is consistent with ZFC that every sequential group
  of sequential order $< \omega_1$ is Fr\'echet and that every
  countable Fr\'echet group is metrizable.
\end{theorem}
\begin{proof}
Let the ground model $V\vdash\text{CH}$, and suppose
$\set{A_\alpha|\alpha\in S_1^2}$ witnesses $\diamondsuit(S_1^2)$ in
$V$. Construct a finite support iteration ${\Bbb P}_{\omega_2} =
\langle\,{\Bbb P}_\alpha,\dot{\Bbb Q}_\alpha:\alpha < \omega_2\,\rangle$
so that whenever $\alpha\in S_1^2$ and $A_\alpha$ codes a ${\Bbb
  P}_\alpha$-name for a strong consequential pair,
$\dot{\Bbb Q}_\alpha$ is a ${\Bbb P}_\alpha$-name for ${\Bbb
  L}_{\nwd^*(\tau)}$. Otherwise, let $\dot{\Bbb Q}_\alpha$ be a ${\Bbb
  P}_\alpha$-name for ${\Bbb L}_{\nwd^*({\Bbb Q})}$. Let ${\bf
  G}_{\omega_2}$ be ${\Bbb P}_{\omega_2}$-generic over $V$.

Assume that in $V[{\bf G}_{\omega_2}]$ there is a sequential group $H$
such that $1 < \so(H) < \omega_1$ or $H$ is countable non-metrizable
and Fr\'echet. The Fr\'echet case is handled almost identically to
\cite{hrusak} so here we only consider the sequential groups of
intermediate sequential orders. In this case there exists a separable
$H$ as above.

Suppose first that in $V[{\bf G}_{\omega_2}]$ there exists a closed
normal subgroup $N\subseteq H$ such that $\psi(H/N)\leq\omega_1$ and
$H/N$ is not metrizable. By Corollary~\ref{loord}, $\so(H/N) <
\omega_1$. If $\so(H/N) = 1$, since $H$ is separable, there is a
non-metrizable countable Fr\'echet group in $V[{\bf
    G}_{\omega_2}]$. Thus we only consider the case of $1 < \so(H/N) <
\omega_1$. As pointed out in \cite{hrusak}, $\frak b = \omega_2$ in
$V[{\bf G}_{\omega_2}]$ so $\psi(H/N)\leq\omega_1 < {\frak b}$.
Corollary~\ref{cscaffs} implies that every scaffold in $H/N$ has a
proper closed subscaffold. Pick a countable dense subgroup $X$ of $H/N$
such that $X$ contains a semiloose $\beta$-scaffold for every
semiloose $\beta$-scaffold in $H/N$ (this only requires adding
countably many countable witnesses to $X$). Assume $X = \omega$ and
put 
$$
\overline{\cal S} = \set{\St\rst X|(S,\St)\text{ is a scaffold in
  }H/N\text{ for some }S\subseteq H/N}.
$$

Now a standard argument implies the existence of a club $C\subset
S_1^2$ relative to $S_1^2$ such that for all $\alpha\in C$, $V[{\bf
    G}_{\alpha}]\models$ $(X,\overline{\cal S}_\alpha)$ is a strong
consequential pair, where $\overline{\cal S}_\alpha = \overline{\cal
  S}\cap V[{\bf G}_\alpha]$ and $X$ extends to a group of intermediate
sequential order.

Alternatively, suppose for every closed normal $N\subseteq H$ such
that $\psi(H/N)\leq\omega_1$ the group $H/N$ is metrizable, i.e.~$H$
is $\omega_1$-collapsible. Just as above, pick a countable dense
subgroup $X$ of $H$ such that $X$ contains a semiloose
$\beta$-scaffold for every semiloose $\beta$-scaffold in $H$ and
assume $X = \omega$. Put 
$$
\overline{\cal S} = 
\set{\St\rst X|(S,\St)\text{ is a scaffold in }H\text{ for some
  }S\subseteq H}
$$ 
and 
$$
\overline{\cal C} = \set{(A,\St)|A\subseteq X,
  \text{$(S,\St)$ is a scaffold such that $\St\rst X\in\overline{\cal
    S}$, $\overline{A}^H\ni\corn S$}}
$$ 
(this $\overline{\cal C}$ codes the closures of subsets of $X$). As
before, conclude that there exists a club $C\subset
S_1^2$ relative to $S_1^2$ such that for all $\alpha\in C$, $V[{\bf
    G}_{\alpha}]\models$ $X\subseteq G$, $(X,\overline{\cal S}_\alpha)$ is a
consequential pair and 
$$
\overline{\cal C}_\alpha = \set{(A,\St)|A\subseteq X,
  \text{$(S,\St)$ is a scaffold such that $\St\rst X\in\overline{\cal
    S}$, $\overline{A}^G\ni\corn S$}}
$$ 
where $\overline{\cal S}_\alpha = \overline{\cal
  S}\cap V[{\bf G}_\alpha]$ and $\overline{\cal C}_\alpha = \overline{\cal
  C}\cap V[{\bf G}_\alpha]$. Embed $X$ in $G$ as in
Definition~\ref{conseqp}. Since the group operation and the
closures of subsets of $X$ in $G$ are `coded' by $\overline{\cal
  S}_\alpha$ and $\overline{\cal C}_\alpha$, the properties above
together with those of $H$, $V[{\bf G}_\alpha]$, and ${\Bbb
  P}_{\omega_2}$ imply that the conditions of
Lemma~\ref{hlindelo} are satisfied and $G$ is hereditarily
Lindel\"of and, therefore, $T_5$. Now Theorem~\ref{hnpsi}
implies that $G$ has a countable pseudocharachter. Applying
Lemma~\ref{cscaffs} shows that $V[{\bf
    G}_{\alpha}]\models$ $(X,\overline{\cal S}_\alpha)$ is a strong
consequential pair.

The choice of $C$ implies that if $\tau$ is the topology of $X$
inherited from $H$ in $V[{\bf G}_{\omega_2}]$ then
for any $\alpha\in C$, $\tau_\alpha = \tau\cap V[{\bf G}_\alpha]$
where $\tau_\alpha$ is the topology on $X$ `induced' by $\overline{\cal S}_\alpha$.

According to Proposition~\ref{pihit}, at some stage
$\alpha\in C$ a set $A_{gen}$ would have been added such that $V[{\bf
    G}_{\alpha+1}]\models$ $A_{gen}\in{\cal I}_{\uG}^+(\tau_\alpha)$ and
the ideal ${\cal I}_{\uG}(\tau_\alpha)\rst A_{gen}$ is
$\omega$-hitting.
Suppose there exists an open neighborhood $U$ of $\uG$ in $X$ such
that $U\cdot U\cap A_{gen} = \emptyset$. Then, just as in
\cite{hrusak}, the set $A = X\setminus U$ is
$A_{gen}$-large. Lemma~\ref{cseq} supplies the conditions necessary
for the conclusion of Lemma~\ref{closowh} to hold. Thus in $V[{\bf
    G}_{\alpha+1}]$ the ideal ${\cal I}_{\uG}^\perp(\tau_\alpha)$ is
$\omega$-hitting w.r.t.~$A_{gen}$ so it follows from Lemmas~\ref{sigpr} and~\ref{fipr}
that ${\cal I}_{\uG}^\perp(\tau_\alpha)$ is $\omega$-hitting
w.r.t.~$A_{gen}$ in $V[{\bf G}_{\omega_2}]$. Hence there is a
$C_1\in[A]^\omega$ such that $C_1$ converges to $\uG$ in
$\tau_\alpha$. Now $C_1\in{\cal S}_\alpha\subseteq\overline{\cal S}_\alpha$
so $C_1$ is a subsequence of $A$ that converges to $\uG$ in $\tau$
contradicting $A\cap U = \emptyset$.

Since $H$ is sequential in $V[{\bf G}_{\omega_2}]$ there exists a
scaffold $(S,\St)$ in $H$ such that $S_{[0]}\subseteq A$ and $\corn S
= \uG$. By Lemma~\ref{pohit} and Lemma~\ref{fsohp} the ideal ${\cal
  I}_{\uG}(\tau_\alpha)\rst A_{gen}$ is $\omega$-hitting in $V[{\bf
    G}_{\omega_2}]$ so there exists an $I\in{\cal
  I}_{\uG}(\tau_\alpha)$ such that $I\cap C_1$ is infinite for every
$C_1\in\St_{[1]}$ contradicting $\corn S = \uG\in\overline{S_{[0]}}$.
\end{proof}

\section{Remarks and open questions}
The results about topological groups with various convergence
properties obtained so far indicate some important implications set
theoretic combinatorics has concerning the existence of such
groups. Much less is known about more subtle interactions between
convergence and group-theoretic properties of the space. It seems
worthwhile to repeat a question asked in \cite{hrusak} here:

\begin{question}\label{hrgqu}Is it consistent with ZFC that some countable topologizable
  group admits a non-metrizable Fr\'echet group topology while another
  does not?
\end{question}

The next question is a recast of Question~\ref{hrgqu} to sequential
groups although it is open for uncountable groups as well. Using the
techniques of \cite{sh} it is possible to construct sequential group
topologies with intermediate sequential orders on any countable
topologizable group using CH.

\begin{question}\label{agiso}Is it consistent with ZFC that some (countable)
  topologizable group admits a sequential topology with intermediate
  sequential order while another does not?
\end{question}

The following two questions do not have any counterparts for Fr\'echet
groups.

\begin{question}Is it consistent with ZFC that groups of intermediate
  sequential order $\alpha$ exist for some $\alpha\in\omega_1$ but not
  for all of them? Only finite $\alpha$? Only infinite ones?
\end{question}

A more specific version of Question~\ref{agiso} asks about the influence
the {\em size} of a group has on its convergence properties.

\begin{question}Is it consistent with ZFC that there is an uncountable
  group of intermediate sequential order but there is no countable
  such?
\end{question}

In~\cite{sha} D.~Shakhmatov repeats his question from 1990
(Question~7.5):

\begin{question}\label{ccsnf}Is a countably compact sequential group
  Fr\'echet-Urysohn?
\end{question}

Here is a short argument showing that the model constructed in this
paper (or, indeed, Hru\v s\'ak-Ramon-Garc\'ia's original model) provides a
consistent positive answer to Question~\ref{ccsnf}.

\begin{lemma}Let $H$ be a countably compact sequential non Fr\'echet
  group in $V[{\bf G}_{\omega_2}]$. Then $H$ is
  $\omega_1$-collapsible.
\end{lemma}
\begin{proof}
By picking a countable subgroup containing an
appropriate witness we can assume that $H$ is separable. If $H$ is not
$\omega_1$-collapsible, there exists a closed normal subgroup
$N\subseteq H$ such that $\psi(H/N) = \omega_1 < {\frak b}$ and $H/N$
is separable, sequential and not metrizable. Since there are no
countable nonmetrizable Fr\'echet groups in $V[{\bf G}_{\omega_2}]$,
the group $H/N$ is not Fr\'echet. Thus $H/N$ contains a semiloose
$2$-scaffold. Corollary~\ref{cscaffs} implies that $H/N$ contains a
closed semiloose $2$-scaffold. But every closed semiloose $2$-scaffold
contains a closed copy of an infinite discrete space contradicting
countable compactness of $H/N$.
\end{proof}

Another standard argument shows that there is a club $C\subset S_1^2$
relative to $S_1^2$ such that for every $\alpha\in C$ the model
$V[{\bf G}_\alpha]$ contains $X\subseteq G$ as in Lemma~\ref{hlindelo}
such that $G$ is sequential, non Fr\'echet, and countably
compact. Lemma~\ref{hlindelo}, together with Theorem~\ref{hnpsi} imply
that such $G$ has a countable pseudocharacter. This leads to a
contradiction, just as in the proof of the lemma above.

To provide some motivation for our final question, recall that a
(countable) space $X$ is {\em analytic} if its topology (viewed as a
set of characteristic functions of open subsets of $X$ in $2^X$
endowed with the standard product opology) is a
continuous image of the irrationals. In \cite{todouzc} S.~Todor\v
cevi\'c and C.~Uzcat\'egui ask the following questions.

\begin{question}\label{soasg}What are the possible sequential orders 
of analytic sequential groups?
\end{question}

\begin{question}\label{ufpnhas}Is there an uncountable family of
  pairwise nonhomeomorphic analytic sequential spaces of sequential
  order $\omega_1$?
\end{question}

The results of this paper show that it is consistent with ZFC that the
only possible sequential orders of analytic sequential groups are $1$
and $\omega_1$ (a free topological group over a convergent sequence is
analytic and has a sequential order of $\omega_1$). Admittedly, such a
consistent answer is somewhat contrary to the spirit of viewing
the results about analytic spaces as `effective' versions of their general
counterparts but it does indicate the only possibility for a `true' ZFC
answer.

To answer~\ref{ufpnhas}, recall that a space is $k_\omega$ if it is a
quotient image of a countable sum of compact spaces. A direct
construction immediately shows that any countable $k_\omega$ (or
indeed, any countable space whose topology is dominated by countably
many first countable subspaces) is analytic (indeed, Borel). Now \cite{zelprot} shows
that there are exactly $\omega_1$ nonhomeomorphic $k_\omega$ countable
group topologies and \cite{shiba} proves that all such topologies
(other than the discrete) have sequential order $\omega_1$. Is this
the best possible result in this direction (at least for group
topologies)? 

As the results above indicate, the model constructed in this paper
tends to trivialize the structure of sequential groups. The final
question is about the classification of countable such groups. As the
answer to Question~\ref{ufpnhas} shows, it has some relevance to the
general structure of analytic sequential groups.

\begin{question}Is it consistent with ZFC that all countable
  sequential groups are $k_\omega$ or metrizable?  Analytic?
\end{question}


\begin{thebibliography}{99}
\bibitem{arh}A.~V.~Arhangel'skii, {\em Relations among the invariants
  of topological groups and their subspaces}, Russian Math.\ Surveys
  {\bf 35} (1980), pp.~1--23

\bibitem{arhtkac}A.~V.~Arhangel'skii and M.~Tkachenko, {\em
  Topological groups and related structures}, Atlantis Studies in
  Mathematics, vol.~I, Atlantis Press/World Scientific,
  Amsterdam-Paris 2008

\bibitem{dowbarman}D.~Barman and A.~Dow, {\em Proper forcing axiom and
  selective separability}, Topology Appl.\ {\bf 159} (2012),
  pp.~806--813

\bibitem{brendle}J.~Brendle and M.~Hru\v s\'ak, {\em Countable
  Fr\'echet Boolean groups: an independence result}, Journal of
  Symb.\ Logic {\bf 74} (2009), no.~3, pp.~1061--1067

\bibitem{dikranjan}D.~Dikranjan, D.~Impieri, and D.~Toller, {\em Metrizability of hereditarily normal
compact like groups}, Rend.\ Istit.\ Mat.\ Univ.\ Trieste {\bf 45}
  (2013), pp.~123--135

\bibitem{dpierone}S.~Dolecki and R.~Peirone, {\em Topological
  semigroups of every countable sequential order}, Recent Developments
  of General Topology and its Applications, Akademie-Verlag, 1992,
  pp.~80--84

\bibitem{dow}A.~Dow, {\em Two classes of Fr\'chet-Urysohn spaces},
  Proc.\ Amer.\ Math.\ Soc.\ {\bf 108} (1990), pp.~241--247

\bibitem{hrusak}M.~Hru\v s\'ak and U.A.~Ramos-Garc\'ia, {\em
  Malykhin's problem}, preprint

\bibitem{kannan}V.~Kannan, {\em Ordinal Invariants in Topology},
  Memoirs of the Amer.\ Math.\ Society {\bf 32} (1981) no.~245

\bibitem{kunen}K.~Kunen, {\em Set theory: An Introduction to
  Independence Proofs}, Vol.~{\bf 102} of Studies in Logic and the
  Foundations of Mathematics, North-Holland, Amsterdam 1980

\bibitem{jech}Thomas J.~Jech{\em Multiple Forcing}, Issue~{\bf 88} of
Cambridge Tracts in Mathematics, Cambridge University Press 1986

\bibitem{nyikos}P.~J.~Nyikos, {\em Metrizability and the
  Fr\'echet-Urysohn property in topological groups},
  Proc.\ Amer.\ Math.\ Soc.\ {\bf 83} (1981), pp.~793--801

\bibitem{noguratanaka}T.~Nogura and Y.~Tanaka, {\em Spaces which
  contain a copy of $S_\omega$ or $S_2$ and their applications},
  Topology Appl.\ {\bf 30} (1988) pp.~51--62

\bibitem{pierone}R.~Peirone, {\em Regular semitopological groups of
  every countable sequential order}, Topology Appl.\ {\bf 58} (1994),
  pp.~145--149

\bibitem{zelprot}I.~Protasov, E.~Zelenyuk, {\em Topologies on groups
  determined by sequences}, VNTL, Lviv 1999

\bibitem{sha}D.~B.~Shakhmatov, {\em Convergence in the presence of
  algebraic structure}, Recent progress in general topology, II,
  North-Holland, Amsterdam 2002, pp.463--484

\bibitem{sha1}D.~B.~Shakhmatov, {\em $\alpha_i$-properties in
  Fr\'echet-Urysohn topological groups}, Topology Proc.\ {\bf 15}
  1990, pp.~143--183

\bibitem{sh}A.~Shibakov, {\em Sequential group topology on rationals
with intermediate sequential order}, Proc.\ Amer.\ Math.\ Soc.\ {\bf
  124} (1996), pp.~2599--2607

\bibitem{shi}A.~Shibakov, {\em Sequential topological groups of any
  sequential order under CH}, Fund.\ Math.\ {\bf 155} (1998) no.~1,
  pp.~79--89

\bibitem{shiba}A.~Shibakov, {\em Metrizability of sequential topological groups
with point-countable $k$-networks}, Proc.\ Amer.\ Math.\ Soc.\ {\bf
  126} (1998), pp.~943--947

\bibitem{todouzc}S.~Todor\v cevi\'c and C.~Uzc\'ategui, {\em Analytic
  $k$-spaces}, Topology Appl.\ {\bf 146--147} (2005), pp.~511--526
\end{thebibliography}
\end{document}